\documentclass[amscd,amssymb,11pt]{article}
\usepackage{amsmath,amssymb,amsthm}
\usepackage{mathrsfs}
\usepackage{fullpage}

\usepackage{amsmath}
\usepackage{amsthm}
\usepackage{amsfonts}
\usepackage{amssymb}
\usepackage{bbm}
\usepackage{color}

\newtheorem{theorem}{Theorem}[section]

\newtheorem{remark}[theorem]{Remark}

\newtheorem{lemma}[theorem]{Lemma}
\newtheorem{prop}[theorem]{Proposition}

\newtheorem{definition}[theorem]{Definition}

\numberwithin{theorem}{section}
\numberwithin{equation}{section}

\usepackage[T1]{fontenc}

\title{{\bf Aleksandrov-Bakelman-Pucci maximum principle for $L^n$-viscosity 
solutions of uniformly elliptic equations with unbounded terms: The critical case}}

\author{
    \textsc{Andrzej \'{S}wi\k{e}ch}\\
 \textit{School of Mathematics, Georgia Institute of Technology}\\
    \textit{Atlanta, GA 30332, U.S.A.}\\
    \textit{E-mail: swiech@math.gatech.edu}}
\date{}

\begin{document}

\maketitle

\begin{abstract} We prove the Aleksandrov-Bakelman-Pucci maximum principle for $L^n$-viscosity 
sub/super-solutions of fully nonlinear uniformly elliptic equations in $\Omega\subset\mathbb R^n$ with the drift coefficient $\gamma$ such that $\gamma(\ln(1+\gamma))^{\frac{n-1}{n}}\in L^n(\Omega)$. The approach is based on convex duality and Trudinger-Moser type inequalities.
\end{abstract}

\qquad

\qquad


\quad


\section{Introduction}\label{sec:intro}
In this paper we prove the Aleksandrov-Bakelman-Pucci (ABP) maximum principle for $L^n$-viscosity 
sub/super-solutions of fully nonlinear uniformly elliptic equations with measurable and unbounded terms in the critical case which was left out in \cite{KS-JMPA}. This amounts to proving the ABP maximum principle for
$L^n$-viscosity subsolutions of the Pucci extremal equations 
\begin{equation}\label{eq:1}
{\mathcal P}^-(D^2u)-\gamma(x)|Du|=f(x)\quad\mbox{in}\,\,\Omega
\end{equation}
where $\Omega$ is a bounded domain in $\mathbb R^n, n\geq 2, f\in L^n(\Omega)$ and $\gamma$ is a nonnegative measurable function on $\Omega$ satisfying an integrability condition which will be described later. We recall that for fixed ellipticity constants $0<\lambda\leq\Lambda$, the Pucci extremal operators ${\mathcal P}^-,{\mathcal P}^+$ are defined for $X\in
\mathbb{S}(n)$, the set of all $n\times n$ symmetric matrices, by
\[
{\mathcal P}^-(X)={\mathcal P}_{\lambda,\Lambda}^-(X):=-\Lambda\mathrm{Tr}(X^+) +\lambda\mathrm{Tr}(X^-), \quad
{\mathcal P}^+(X)={\mathcal P}_{\lambda,\Lambda}^+(X):=\Lambda\mathrm{Tr}(X^-) -\lambda\mathrm{Tr}(X^+),
\]
where $X^+$ and $X^-$ are the positive and negative parts of $X$. We remark that ${\mathcal P}^+(X)=-{\mathcal P}^-(-X)$.

Classical ABP maximum principle is formulated for strong sub/super-solutions of linear equations which do not have to be uniformly elliptic. It originated with works of Aleksandrov, Bakelman and Pucci \cite{A1,A2,Bak,Pucci}. We refer to \cite[Theorem 9.1]{GT} for the full statement. We recall it here in a version directly applicable to fully nonlinear uniformly elliptic equations in Proposition \ref{prop:ABPnLp}. ABP maximum principle allows to estimate the maximum value of a strong subsolution $u\in W^{2,n}_{\rm loc}(\Omega)\cap C(\overline \Omega)$ of \eqref{eq:1} by the maximum value of $u$ over $\partial\Omega$ plus a constant times the $L^n$ norm of $f$ over the so-called upper contact set of $u$. 
Various generalizations of the ABP maximum principle followed from the fundamental work \cite{fs} on improved integrability of Green's functions. These versions replace the $L^n$ norm of $f$ over the upper contact set of $u$ by the $L^p$ norm of $f$ over the whole $\Omega$, however they allow $p$ to be slightly less than $n$. They include results in \cite{Cab} where $\gamma$ is bounded, \cite{FokT,Fok} where $\gamma\in L^{q}(\Omega)$ for $q>n$ and \cite{KR3,JungWoo} where $\gamma\in L^n(\Omega)$.

The first proofs of ABP maximum principle in the original version with the upper contact set for viscosity solutions of uniformly elliptic equations without discontinuous terms appeared in \cite{Caf,CC,T1}. The result was then proved for $L^n$-viscosity 
subsolutions of \eqref{eq:1} when $\gamma$ is bounded, i.e. when gamma is a positive constant, in \cite{CCKS}. Generalizations to $L^p$-viscosity subsolutions when $f\in L^p(\Omega)$ for $p_0<p<n$ for some $p_0<n$ but without the upper contact set appeared first in \cite{CCKS}, with a slightly sharpened version in \cite{cs}, when $\gamma$ is bounded, and then the same result was proved when $\gamma\in L^q_+(\Omega), q>n$ in \cite{FokT,Fok}, see also \cite{KS-MathAnn,KS-JMSJ}. In \cite{KS-JMPA} such generalized ABP maximum principle was established for $L^p$-viscosity subsolutions when $\gamma\in L^n_+(\Omega),f\in L^p(\Omega)$ for $p_1<p<n$ for some $p_1<n$. 

The ABP maximum principle for $L^p$-viscosity subsolutions of \eqref{eq:1} in the original version with the upper contact set when $\gamma$ is unbounded was studied in the recent paper \cite{KS-JMPA}. Two results were obtained there. The first established that the ABP maximum principle holds for $L^n$-viscosity 
subsolutions of \eqref{eq:1} if $\gamma\in L^q_+(\Omega), q>n$ and $f\in L^n(\Omega)$. The second guaranteed that it also holds for $L^p$-viscosity 
subsolutions of \eqref{eq:1} if $\gamma\in L^n_+(\Omega)$ and $f\in L^n(\Omega)$, provided that $p<n$. The critical case $p=n$, which is the $L^n$-viscosity solution equivalent of the classical ABP maximum principle remained open. The reason why this case is difficult is that the theory of $L^p$-viscosity solutions and the proof of ABP maximum principle crucially depend on the ability to solve Pucci extremal equations \eqref{eq:1}. Thus in the critical case $\gamma\in L^n_+(\Omega)$ and $f\in L^n(\Omega)$, we need to be able to find a strong solution $u\in W^{2,n}_{\rm loc}(\Omega)\cap C(\overline \Omega)$ of \eqref{eq:1} with a prescribed value on the boundary. However if $u\in W^{2,n}(\Omega)$ then the term $\gamma|Du|$ may not belong to $L^n(\Omega)$ and the question of existence of strong solutions in $W^{2,n}_{\rm loc}(\Omega)$ remains open. Only existence of strong solutions in $u\in W^{2,n}_{\rm loc}(\Omega)\cap C(\overline \Omega)$ for $p<n$ can be proved. The idea is then that if we can guarantee that $\gamma$ is such that $\gamma|Du|\in L^n(\Omega)$ for $u\in W^{2,n}(\Omega)$ then we should be able to solve the extremal equation \eqref{eq:1} in the class of strong solutions in $W^{2,n}_{\rm loc}(\Omega)$ and this should open up a possibility of proving the ABP maximum principle. Convex duality and Trudinger-Moser inequality suggest that the proper condition for $\gamma$ is $\gamma(\ln(1+\gamma))^{\frac{n-1}{n}}\in L^n(\Omega)$. Indeed we prove in Theorem \ref{prop:ABPLnviscosity} that the ABP maximum principle for $L^n$-viscosity subsolutions of \eqref{eq:1} holds under this condition. The strategy of the proof is similar to the proofs of the corresponding results in \cite{KS-JMPA}. We use solvability of extremal equations to perturb an $L^n$-viscosity subsolution $u$ of \eqref{eq:1} to an $L^n$-viscosity subsolutionn $u_m$ of \eqref{eq:1} with $\gamma, f$ replaced by proper approximating smooth functions $\gamma_m,f_m$ and additional nonlinear gradient terms. We then approximate $u_m$ by sup-convolutions and mollify them to obtain smooth subsolutions of approximating equations with extra gradient terms for which we can adapt the classical proof of the ABP maximum principle. 

In this manuscript we focus entirely on the ABP maximum principle. However ,once the maximum principle is known it should be possible to develop the full theory of $L^n$-viscosity solutions in this case along the lines of how it was done in \cite{KS-JMPA} for the cases considered there.

Finally, we mention that there are many other
ABP maximum principle type results available in the literature for a wide range of cases. They cover for instance equations with superlinearly growing gradient terms \cite{KS-MathAnn, KS-JMSJ, KS-JFPTA, NAK}, 
degenerate/singular equations \cite{ACP,BFO,CDePDiCM, DFQ, Im, Junges}, equations in unbounded domains \cite{ARV, BCV, CD, CDLV, CDV, V, V1,V2}, and many other equations, including integro-PDE.

\section{Notation and Preliminaries}\label{sec:prelim}
Unless stated otherwise, in this manuscript $\Omega$ is a bounded domain in $\mathbb R^n, n\geq 2$.
If $x\in\mathbb{R}^n,r>0$, we denote by $B_r(x)$ the open ball centered at $x$ with radius $r$ and if $x=0$ we will write $B_r$.
For $r\in\mathbb{R}, x\in\mathbb{R}^n, X\in \mathbb{S}(n)$, we write $|r|,|x|$ and $\|X\|$ to denote the standard norms in $\mathbb{R},
\mathbb{R}^n$ and $\mathbb{S}(n)$. If $V$ is a subset of $\mathbb{R}^n$, we write ${\rm diam}(V)$ to denote its diameter and
$|V|$ to denote its Lebesgue measure (if $V$ is measurable). The notation $V\Subset U$ means that $\overline V$ (the closure of $V$)
is a compact subset of $U$. We denote by ${\rm dist}(V,U)$ the distance between $V$ 
and $U$. For a function $f$ we write $f^+$ and $f^{-}$ to denote the positive and negative parts of $f$.

For a function $f\in L^n(U)$ we will denote its norm by $\|f\|_{L^n(U)}$. $L^n_+(U)$ is the set of nonnegative functions in $L^p(U)$. The notation $W^{1,n}(U),W^{1,n}_0(U), W^{2,n}(U)$ is used to denote the standard Sobolev spaces equipped with the norms
$\|u\|_{W^{1,p}(U)}, \|u\|_{W^{2,p}(U)}$. We write $C(U), BUC(U)$ to denote respectively the spaces of continuous and bounded and uniformly continuous functions on $U$. If $U$ is bounded we write 
$\|u\|_{C(\overline U)}$ for the norm of $u$ in $C(\overline U)$. 

If $u\in C(\Omega)$ for an open set $\Omega$, the upper contact set of $u$ is defined as
$$
\Gamma^+(u)=\Gamma^+(u,\Omega):=\{ x\in \Omega \ | \ \exists p\in \mathbb{R}^n\mbox{ such that }u(y)\leq u(x)+\langle p,y-x\rangle \mbox{ for }y\in \Omega\}.
$$

We denote
\[
\phi(t)=\exp\left(t^{\frac{n}{n-1}}\right)-1,\quad t\geq 0,
\]
and let $\psi:[0,+\infty)\to\mathbb R$ be its Legendre transform, that is $\psi(s)=\sup\{st-\phi(t):t\geq 0\}$.
It is well known that $\psi$ behaves asymptotically as $s\to +\infty$ like
\[
s(\ln(1+s))^{\frac{n-1}{n}}.
\]
Therefore, there exists a constant $C>0$ such that for all $t,s\geq 0$,
\begin{equation}\label{eq:dual}
st\leq \exp\left(t^{\frac{n}{n-1}}\right)-1+C\left(t+s(\ln(1+s))^{\frac{n-1}{n}}\right).
\end{equation}
We denote $\Psi(s)=s^n(\ln(1+s))^{n-1}$ and define 
\[
L^\Psi(\Omega):=\{\mbox{measurable}\,f:\Omega\to\mathbb R: \int_{\Omega}\Psi(|f(x)|)dx<+\infty\}.
\]
We denote $L^\Psi_+(\Omega)$ to be the set of nonnegative functions in $L^\Psi(\Omega)$. We notice that $L^n(\Omega)\subset L^\Psi(\Omega)$.

We will need two Trudinger-Moser embedding type results. The first was proved in \cite{TRemb} with the sharp constant obtained in \cite{Moser}.
\begin{theorem}
\label{th:MT}
Let $\Omega$ be a bounded domain. Let $\alpha_n=n\sigma_{n-1}^{\frac{1}{n-1}}$, where $\sigma_{n-1}$ is the measure of the unit sphere in $\mathbb R^n$. There exists a constant $C_n>0$ such that for all $u\in W^{1,n}_0(\Omega)$
\begin{equation}\label{eq:MT}
\int_\Omega\exp\left(\alpha_n \left(\frac{|u|}{\|Du\|_{L^{n}(\Omega)}}\right)^{\frac{n}{n-1}}\right) dx\leq C_n|\Omega|.
\end{equation}
\end{theorem}
The second result is a version of Trudinger-Moser embedding for functions in $W^{1,n}(\Omega)$. The version we state below can be easily deduced from Theorem 1.1 of \cite{LiRuf}.
\begin{theorem}
\label{th:MT1}If $\Omega$ is a bounded domain with smooth boundary then 
\begin{equation}\label{eq:MT1}
\sup_{\|u\|_{W^{1,n}(\Omega)\leq 1}}\int_\Omega\exp\left(\alpha_n |u|^{\frac{n}{n-1}}\right)dx=C(n,\Omega)<+\infty.
\end{equation}
\end{theorem}
We refer the reader to the monographs \cite{KR,RaoRen,RGMP} for more on Orlicz spaces, including duality in Orlicz spaces.

Below is the classical ABP maximum principle for strong solutions which we state here for subsolutions and supersolutions of extremal equations. This version follows directly from the standard version \cite[Theorem 9.1]{GT} by linearization (see e.g the proof of Theorem 3.1 of \cite{KS-JMPA}).
\begin{theorem}\label{prop:ABPnLp}
Let $\Omega$ be a bounded domain, $\gamma\in L^n_+(\Omega), f\in L^n(\Omega)$. There exist $C=C(n,\lambda,\|\gamma\|_{L^n(\Omega)})$ such that if
$u\in W^{2,n}_{\rm loc}(\Omega)\cap C(\overline \Omega)$ is a strong subsolution (resp., supersolution) of 
\begin{equation}\label{eq:2n-p}
{\mathcal P}^-(D^2u)-\gamma(x)|Du|=f(x) \,\,\,\mbox{in}\,\,\Omega
\end{equation}
\[
\left(\mbox{resp.,}\,\,
{\mathcal P}^+(D^2u)+\gamma(x)|Du|=-f(x) \,\,\,\mbox{in}\,\,\Omega
\right)
\]
then
\begin{equation}\label{eq:ABPnstatementLp}
\max_{\overline \Omega} u\leq \max_{\partial\Omega} u+C{\rm diam}(\Omega)\|f\|_{L^n(\Gamma^+(u))}
\end{equation}
\[
\left(\mbox{resp.,}\,\,
\min_{\overline \Omega} u\geq \min_{\partial\Omega} u-C{\rm diam}(\Omega)\|f\|_{L^n(\Gamma^+(-u))}\right).
\]
\end{theorem}
\begin{remark}\label{rem:ABPnLp}
We remark that the constant $C$ in \eqref{eq:ABPnstatementLp} is such that $C(n,\cdot,\cdot)$ is continuous on $(0,\infty)\times[0,\infty)$ and for every fixed $n,p$, $C(n,\lambda,p)\to+\infty$ as $\lambda\to 0$.
\end{remark}

We recall the definition of an $L^p$-viscosity solution of from \cite{CCKS}.

\begin{definition}
\label{defvisc}
A function $u\in C(\Omega)$ is an $L^p$-viscosity subsolution (resp., supersolution) of $F(x,u,Du,D^2u)=f(x)$ in $\Omega$ if
$$
{\rm ess}\liminf_{y\to x}\left(F(y,u(y),D\varphi (y),D^2\varphi (y))- f(y)\right) \leq 0
$$
$$
\left(\mbox{resp., }{\rm ess}\limsup_{y\to x}\left(F(y,u(y),D\varphi (y),D^2\varphi (y))- f(y)\right)
\geq 0\right)
$$
whenever for $\varphi\in W^{2,p}_{\rm loc}(\Omega)$, $u-\varphi$ attains a local
maximum $($resp., minimum$)$ at $x\in \Omega$.
Finally, $u\in C(\Omega)$ is called an $L^p$-viscosity solution 
if it is both an $L^p$-viscosity subsolution and an $L^p$-viscosity supersolution.
\end{definition}

If in the above definition we assume that $F,f$ are continuous and require that $\varphi\in C^2(\Omega)$, the function $u$ is called
a $C$-viscosity subsolution (resp., supersolution). The ${\rm ess}\liminf$ and ${\rm ess}\limsup$ can then be removed and Definition \ref{defvisc} becomes the standard pointwise definition of a viscosity solution, see \cite{CIL}.

\section{Existence of solutions to Pucci extremal equations}\label{sec:Pucci}

\begin{lemma}\label{lem:strsolestimates}
Let $\Omega$ be a bounded domain, $f\in L^n(\Omega),\gamma\in L^\Psi_+(\Omega)$ and $g\in C(\partial\Omega)$. Let $u\in W^{2,n}_{\rm loc}(\Omega)\cap C(\overline \Omega)$ be a strong solution of
\begin{equation}\label{eq:Puccieq1}
\begin{cases}
{\mathcal P}^{\pm}(D^2u)\pm\gamma(x)|Du|=f(x)\quad \mbox{in }\Omega,\\
u=g\quad\mbox{on }\partial \Omega.
\end{cases}
\end{equation}
Then for every
$\Omega' \Subset\Omega$, 
\begin{equation}\label{lem:est1}
\|u\|_{W^{2,n}(\Omega')}
\le
C(n,\lambda,\Lambda,\gamma,\|\gamma\|_{L^n(\Omega)},
{\rm dist}(\Omega',\partial\Omega), {\rm diam}(\Omega))
\left(
\|g\|_{C(\partial\Omega)}
+
\|f\|_{L^n(\Omega)}
\right).
\end{equation}
The dependence of $C$ in \eqref{lem:est1} on $\gamma$ is determined by how fast integrals of $\gamma^n(\ln(1+\gamma))^{n-1}$ over balls of radius $R$ converge to $0$ as $R\to 0$.
\end{lemma}

\begin{proof}

We will only do the proof for ${\mathcal P}^-(D^2u)-\gamma(x)|Du|=f(x)$.
Let $\Omega' \Subset\Omega$. Since we can cover $\Omega'$ by a finite number of balls having fixed radius
$R$, it is enough to show \eqref{lem:est1} for $B_R$ instead of $\Omega'$.
We will denote the measure of $B_R$ by $|B_R|=\omega_n R^n$, where $\omega_n$
is the measure of $B_1$. Let $\rho \in (0,1)$ and $\eta \in C_0^2(B_R)$
be such that $0 \le \eta \le 1$, $\eta=1$ in $B_{\rho R}$,
$\eta=0$ for $|x|\ge \tilde{\rho}R$, where $\tilde{\rho}=(1+\rho)/2$,
and
\[
|D\eta|
\le
\frac{c}{(1-\rho)R},
\qquad
\|D^2\eta\|
\le
\frac{c}{(1-\rho)^2R^2}
\]
for some absolutr constant $c$.
We set $v=\eta u$. Then $v\in W^{2,n}_0(B_R)$ and ${\rm supp}(v)\subset B_{\tilde{\rho}R}$. Therefore, using standard elliptic estimates \cite{Caf,CC,e} , we have
\[
\|v\|_{W^{2,n}(B_{\tilde{\rho}R})}
\le
C_1 \|P^{-}(D^2v)\|_{L^n(B_{\tilde{\rho}R})},
\]
which implies by Theorem \ref{th:MT} that for every $i=1,...,n$,
\begin{equation}\label{eq:MTa}
\int_{B_{\tilde{\rho}R}}\exp\left(\alpha_n \left(\frac{|v_{x_i}|}{\|D^2v\|_{L^{n}(B_{\tilde{\rho}R})}}\right)^{\frac{n}{n-1}}\right) dx\leq \bar C_nR^n
\end{equation}
for some constant $\bar C_n$.

Hence we have
\begin{align}\label{eq:a1}
\|D^2u\|_{L^n(B_{\rho R})}
&\le
\|D^2v\|_{L^n(B_{\tilde{\rho}R})}
\nonumber\\
&\le
C_1 
\|P^{-}(D^2v)\|_{L^n(B_{\tilde{\rho}R})}
\nonumber\\
&=
C_1 
\|P^{-}(\eta D^2u + 2D\eta \otimes Du + u D^2\eta)\|_{L^n(B_{\tilde{\rho}R})}
\\
&\le
C_2
\left(
\|\eta P^{-}(D^2u)\|_{L^n(B_{\tilde{\rho}R})}
+
\frac{1}{(1-\rho)R}\|Du\|_{L^n(B_{\tilde{\rho}R})}
+
\frac{1}{(1-\rho)^2R^2}\|u\|_{L^n(B_{\tilde{\rho}R})}
\right).\nonumber
\end{align}

It follows from the equation and the maximum principle that
\begin{align}\label{eq:a2}
C_2 \|\eta P^{-}(D^2u)\|_{L^n(B_{\tilde{\rho}R})}
&\le
C_2 \|f\|_{L^n(B_{\tilde{\rho}R})}
+
C_2 \|\eta \gamma Du\|_{L^n(B_{\tilde{\rho}R})}
\nonumber\\
&\le
C_2 \|f\|_{L^n(B_{\tilde{\rho}R})}
+
C_2 \|\gamma D(\eta u)\|_{L^n(B_{\tilde{\rho}R})}
\\
&\quad +
C_3 \|\gamma\|_{L^n(B_{\tilde{\rho}R})}
\frac{
\|g\|_{C(\partial\Omega)}
+
\|f\|_{L^n(\Omega)}
}{
(1-\rho)R
}.\nonumber
\end{align}

Using \eqref{eq:dual} we now have for $\varepsilon>0$
\begin{align}
 \gamma |v_{x_i}|&=\|D^2v\|_{L^{n}(B_{\tilde{\rho}R})} \frac{\gamma}{\varepsilon}\frac{\varepsilon|v_{x_i}|}{\|D^2v\|_{L^{n}(B_{\tilde{\rho}R})}}\leq
\|D^2v\|_{L^{n}(B_{\tilde{\rho}R})}
\left(\exp\left(\left(\frac{\varepsilon|v_{x_i}|}{\|D^2v\|_{L^{n}(B_{\tilde{\rho}R})}}\right)^{\frac{n}{n-1}}\right)-1\right.
\nonumber\\
&+\left.C\left(\frac{\varepsilon|v_{x_i}|}{\|D^2v\|_{L^{n}(B_{\tilde{\rho}R})}}+\frac{\gamma}{\varepsilon}\left(\ln\left(1+\frac{\gamma}{\varepsilon}\right)\right)^{\frac{n-1}{n}}\right)\right).
\nonumber
\end{align}
Therefore,
\begin{align}
 (\gamma |v_{x_i}|)^n
 &\leq
C_4\|D^2v\|_{L^{n}(B_{\tilde{\rho}R})}^n
\left(\exp\left(\left(\frac{n\varepsilon|v_{x_i}|}{\|D^2v\|_{L^{n}(B_{\tilde{\rho}R})}}\right)^{\frac{n}{n-1}}\right)\right.
\nonumber\\
&+\left.\frac{\varepsilon^n|v_{x_i}|^n}{\|D^2v\|_{L^{n}(B_{\tilde{\rho}R})}^n}+\frac{\gamma^n}{\varepsilon^n}\left(\ln\left(1+\frac{\gamma}{\varepsilon}\right)\right)^{n-1}\right).
\nonumber
\end{align}
Choosing $\varepsilon$ such that $n\varepsilon^{\frac{n}{n-1}}\leq\alpha_n$ and integrating the above inequality we thus have
\begin{align}
\int_{B_{\tilde{\rho}R}} (\gamma |v_{x_i}|)^n dx&\leq C_5\|D^2v\|_{L^{n}(B_{\tilde{\rho}R})}^n\Bigg(\bar C_n\omega_nR^n+\varepsilon^n
 \nonumber
 \\
 &
+\int_{B_{\tilde{\rho}R}}\frac{\gamma^n}{\varepsilon^n}\left(\ln\left(1+\frac{\gamma}{\varepsilon}\right)\right)^{n-1}dx\Bigg)=:C(\varepsilon,R)\|D^2v\|_{L^{n}(B_{\tilde{\rho}R})}^n.
\nonumber
\end{align}
We now choose first $\varepsilon$ sufficiently small and then take $R=R(\varepsilon)$ sufficiently small so that $C(\varepsilon,R(\varepsilon))\leq (2C_2n)^{-n}$. We remark that the dependence of $R(\varepsilon)$ on $\gamma$ is is determined by how fast integrals of $\gamma^n(\ln(1+\gamma))^{n-1}$ over balls of radius $R$ converge to $0$. We now have
\begin{equation}\label{eq:a3}
\|\gamma Dv\|_{L^n(B_{\tilde{\rho}R})}\leq \sum_{i=1}^n \|\gamma v_{x_i}\|_{L^n(B_{\tilde{\rho}R})}\leq \frac{1}{2C_2}\|D^2v\|_{L^{n}(B_{\tilde{\rho}R})}\leq
\frac{C_1}{2C_2}\|P^{-}(D^2(v))\|_{L^n(B_{\tilde{\rho}R})}.
\end{equation}

It follows from \eqref{eq:a1}, \eqref{eq:a2} and \eqref{eq:a3} that for some constant $C_6$
\begin{align*}
&\frac{C_1}{2}
\|P^{-}(D^2(\eta u))\|_{L^n(B_{\tilde{\rho}R})}
\le
C_2 \|f\|_{L^n(\Omega)}
\\
&\qquad\quad+
\frac{C_6}{(1-\rho)^2R^2}
\left(
(1-\rho)R
\bigl(
\|g\|_{C(\partial\Omega)}
+
\|f\|_{L^n(\Omega)}
+
\|Du\|_{L^n(B_{\tilde{\rho}R})}
\bigr)
+
\|u\|_{L^n(B_{\tilde{\rho}R})}
\right).
\end{align*}

Hence, again from \eqref{eq:a1}, we obtain
\begin{equation}\label{eq:a4}
(1-\rho)^2R^2
\|D^2u\|_{L^n(B_{\rho R})}
\le
C_7
\left(
\|g\|_{C(\partial\Omega)}
+
\|f\|_{L^n(\Omega)}
\right)
+
C_8
\left(
(1-\tilde{\rho})R
\|Du\|_{L^n(B_{\tilde{\rho}R})}
+
\|u\|_{L^n(B_{\tilde{\rho}R})}
\right).
\end{equation}

If we introduce norms
\[
\Psi_k(u)
=
\sup_{0<\rho<1}
(1-\rho)^k R^k
\|D^k u\|_{L^n(B_{\rho R})},
\qquad
k=0,1,2,
\]
then \eqref{eq:a4} gives the inequality
\[
\Psi_2(u)
\le
C_7
\left(
\|g\|_{C(\partial\Omega)}
+
\|f\|_{L^n(\Omega)}
\right)
+
C_8\bigl(\Psi_1(u)+\Psi_0(u)\bigr).
\]

The required estimate follows from the interpolation inequality
\[
\Psi_1 \le \delta\,\Psi_2 + \frac{C}{\delta}\,\Psi_0,
\]
for $\delta>0$, which may be found in \cite{GT}.
\end{proof}

\begin{theorem}\label{th:strsol}
Let $\Omega$ be a bounded domain which satisfies a uniform exterior cone condition. Let $f\in L^n(\Omega),\gamma\in L^\Psi_+(\Omega), {\rm supp}(\gamma)\Subset\Omega$ and $g\in C(\partial\Omega)$. Then there exists a unique strong solution $u\in W^{2,n}_{\rm loc}(\Omega)\cap C(\overline \Omega)$ of
\begin{equation}\label{eq:Puccistrong}
\begin{cases}
{\mathcal P}^{\pm}(D^2u)\pm\gamma(x)|Du|=f(x)\quad \mbox{in }\Omega,\\
u=g\quad\mbox{on }\partial \Omega.
\end{cases}
\end{equation}
\end{theorem}

\begin{proof}
We will only do the proof for ${\mathcal P}^-(D^2u)-\gamma(x)|Du|=f(x)$. Let $\Omega_1$ be an open domain with smooth boundary such that
\[
{\rm supp}(\gamma)\Subset\Omega_1\Subset\Omega.
\]
We define
\[
B:=\{u\in W^{1,n}(\Omega_1)\cap C(\overline\Omega): \|u\|_{W^{2,n}(\Omega_1)}+\|u\|_{C(\overline\Omega)}\leq K, u=g\,\,\mbox{on}\,\,\partial \Omega\},
\]
where $K>0$ will be determined later. We remind that in the space $W^{1,n}(\Omega_1)\cap C(\overline\Omega)$ we consider the norm $\|u\|_{W^{1,n}(\Omega_1)}+\|u\|_{C(\overline\Omega)}$. Denote for $m=1,2,...$, $\gamma_m:=\min(\gamma,m)$ and $\tilde\gamma_m:=\gamma-\gamma_m$. For $u\in B$ we consider the equation
\[
\begin{cases}
{\mathcal P}^-(D^2v)-\gamma_m(x)|Dv|=f(x)+\tilde\gamma_m(x)|Du(x)|\quad \mbox{in }\Omega,\\
v=g\quad\mbox{on }\partial \Omega.
\end{cases}
\]
The calculation in Step 1 below shows that $\tilde\gamma_m|Du|\in L^n(\Omega)$. Let $v\in W^{2,n}_{\rm loc}(\Omega)\cap C(\overline\Omega)$ be the unique strong solution of the above equation. We denote it by $Tu$. A fixed point of $T$ will be a solution of \eqref{eq:Puccistrong}. To show the existence of a fixed point we will use the Schauder Fixed Point Theorem. Obviously the set $B$ is closed and convex.

{\bf Step 1.} $T:B\to B$ for an appropriate $K$.

Denote $a_n=(\frac{\alpha_n}{n})^{\frac{n-1}{n}}$. Similarly as in the proof of Theorem \ref{lem:strsolestimates} and using that $\tilde\gamma_m(x)=0$ if $\gamma_m(x)\leq m$ (which eliminates the need for the term $Ct$ in \eqref{eq:dual}), writing for $\varepsilon>0$
\[
 \tilde\gamma_m |u_{x_i}|=\frac{K\varepsilon}{a_n}\frac{\tilde\gamma_m}{\varepsilon}\frac{a_n|u_{x_i}|}{K},
 \]
we have for $\varepsilon>0$
\[
 (\tilde\gamma_m |u_{x_i}|)^n
 \leq
C\frac{K^n\varepsilon^n}{a_n^n}
\left(\exp\left(\alpha_n\left(\frac{|u_{x_i}|}{K}\right)^{\frac{n}{n-1}}\right)
+\frac{\tilde\gamma_m^n}{\varepsilon^n}\left(\ln\left(1+\frac{\tilde\gamma_m}{\varepsilon}\right)\right)^{n-1}\right).
\]
Integrating over $\Omega_1$ and using \eqref{eq:MT1}, we now obtain
\[
\int_{\Omega_1} (\tilde\gamma_m |u_{x_i}|)^ndx\leq C\frac{K^n\varepsilon^n}{a_n^n}\left(C(n,\Omega_1)+\int_{\Omega_1}\frac{\tilde\gamma_m^n}{\varepsilon^n}\left(\ln\left(1+\frac{\tilde\gamma_m}{\varepsilon}\right)\right)^{n-1}dx\right).
\]
We notice that the right hand side above can be made arbitrarily small by first choosing $\varepsilon$ small enough and then choosing $m$ large enough. Therefore, we showed that if $m:=m(K)$ is large enough, then for every $u\in B$
\begin{equation}\label{eq:a5}
\|\tilde\gamma_m|Du|\|_{L^n(\Omega_1)}\leq 1.
\end{equation}
Let $C_1:=C(n,\lambda,\Lambda,\gamma,\|\gamma\|_{L^n(\Omega)},
{\rm dist}(\Omega_1\partial\Omega), {\rm diam}(\Omega))$ be the constant from \eqref{lem:est1} and $C_2:=C(n,\lambda,\\
\|\gamma\|_{L^n(\Omega)}){\rm diam}(\Omega)$ be the constant from Theorem \ref{prop:ABPnLp}. It now follows from \eqref{lem:est1} and Theorem \ref{prop:ABPnLp} that for every $K$, if $m$ is sufficiently large then
\[
\|v\|_{W^{2,n}(\Omega_1)}+\|v\|_{C(\overline\Omega)}\leq (1+C_1+C_2)(1+\|f\|_{L^n(\Omega)}+\|g\|_{L^\infty(\partial\Omega)}).
\]
Hence for $K:=(1+C_1+C_2)(1+\|f\|_{L^n(\Omega)}+\|g\|_{L^\infty(\partial\Omega)})$, $T:B\to B$.

{\bf Step 2.} $T$ is continuous.

Suppose $u,u_k\in B, k=1,2,...$ and $\|u-u_k\|_{W^{1,n}(\Omega_1)}+\|u-u_k\|_{C(\overline\Omega)}\to 0$ as $k\to\infty$. 
Since $Tu_k\in B$, we have that, up to a subsequence, $Tu_k\rightharpoonup w$ in $W^{2,n}(\Omega_1)$ for some $w$ such that $\|w\|_{W^{2,n}(\Omega_1)}\leq K$. Using compactness of the Sobolev embedding, this implies that $\|Tu_k-w\|_{W^{1,n}(\Omega_1)}\to 0$ as $k\to\infty$. To conclude, it is enough to show that
$\|Tu_k-Tu\|_{C(\overline\Omega)}\to 0$. This will imply that $Tu$ restricted to $\Omega_1$ is equal to $w$.

The convergence of $\|Tu_k-Tu\|_{C(\overline\Omega)}$ will follow from Theorem \ref{prop:ABPnLp} once we establish that the $L^n(\Omega_1)$ norm of
$\tilde\gamma_m(x)(|Du_k(x)|-|Du(x)|)$ converges to $0$. Denote $w^k:=Du_k-Du$. We have $|\tilde\gamma_m(x)(|Du_k(x)|-|Du(x)|)|\leq 
\tilde\gamma_m(x)|Dw^k(x)|$. We now set $a_n=(\frac{\alpha_n}{2n})^{\frac{n-1}{n}}$ and write for $\epsilon>0$
\[
 \tilde\gamma_m |w^k_{x_i}|=\frac{K}{\varepsilon a_n}\tilde\gamma_m\varepsilon\frac{a_n|v_{x_i}|}{K},
 \]
 and estimate similarly to what we did in Step 1 to get,
 \[
 (\tilde\gamma_m |w^k_{x_i}|)^n
 \leq
C\frac{K^n}{\varepsilon^n a_n^n}
\left(\left(\exp\left(\frac{a_n|w^k_{x_i}|}{K}\right)^{\frac{n}{n-1}}-1\right)^n
+(\tilde\gamma_m\varepsilon)^n\left(\ln\left(1+\tilde\gamma_m\varepsilon\right)\right)^{n-1}\right).
\]
Since $(e^t-1)^n\leq e^{nt}-1$ for $t\geq 0$, we obtain,
\[
 (\tilde\gamma_m |w^k_{x_i}|)^n
 \leq
C\frac{K^n}{\varepsilon^n a_n^n}
\left(\exp\left(\frac{\alpha_n}{2}\left(\frac{|w^k_{x_i}|}{K}\right)^{\frac{n}{n-1}}\right)-1
+(\tilde\gamma_m\varepsilon)^n\left(\ln\left(1+\tilde\gamma_m\varepsilon\right)\right)^{n-1}\right).
\]
Integrating we now have
\begin{align*}
\int_{\Omega_1} (\tilde\gamma_m |w^k_{x_i}|)^ndx
&\leq 
C\frac{K^n}{\varepsilon^n a_n^n}
\int_{\Omega_1} \left(\exp\left(\frac{\alpha_n}{2}\left(\frac{|w^k_{x_i}|}{K}\right)^{\frac{n}{n-1}}\right)-1\right)dx
\\
&+ C\frac{K^n}{a_n^n}\int_{\Omega_1} 
\tilde\gamma_m^n\left(\ln\left(1+\tilde\gamma_m\varepsilon\right)\right)^{n-1}dx
\end{align*}
The second integral of the right hand side above converges to $0$ as $\varepsilon\to 0$ by the dominated convergence theorem. Regarding the first integral, we notice that the functions
\[
h_k(x):=\exp\left(\frac{\alpha_n}{2}\left(\frac{|w^k_{x_i}(x)|}{K}\right)^{\frac{n}{n-1}}\right)-1
\]
converge (up to a subsequence) to $0$ a.e. and are uniformly integrable as by \eqref{eq:MT1}, $\int_{\Omega_1} h_k^2 dx<C$ for some $C$ and all $k\geq 1$. Thus, if we denote $b_k:=\int_{\Omega_1} h_k(x) dx$ and take $\varepsilon_k:=b_k^{\frac{1}{2n}}$ then
\[
C\frac{K^n}{\varepsilon_k^n a_n^n}
\int_{\Omega_1} h_k(x)dx=C\frac{K^n}{a_n^n}b_k^{\frac{1}{2}}\to 0\quad\mbox{as}\,\,k\to\infty.
\]
Since the above arguments apply for every subsequence of $u_k$, we conclude that $\|Tu_k-w\|_{W^{1,n}(\Omega_1)}+\|Tu_k-Tu\|_{C(\overline\Omega)}\to 0$ along the whole sequence.

{\bf Step 3.} $T(B)$ is compact.

Let $u_k\in B, k=1,2,...$ Then for some subsequence, still denoted by $u_k$, we have $u_k\rightharpoonup \bar u,Tu_k\rightharpoonup w$ in $W^{2,n}(\Omega_1)$ for some $\bar u,w$ such that $\|\bar u\|_{W^{2,n}(\Omega_1)}+\|\bar u\|_{C(\overline\Omega_1)}\leq K$ and $\|w\|_{W^{2,n}(\Omega_1)}+\|w\|_{C(\overline\Omega_1)}\leq K$. Thus, by compactness of the Sobolev embedding, $\|u_k-\bar u\|_{W^{1,n}(\Omega_1)}+\|Tu_k-w\|_{W^{1,n}(\Omega_1)}\to 0$ as $k\to\infty$. We extend $\bar u$ to a function in $B$. We can now repeat the proof from Step 2 to show that $T\bar u=w$ and $\|Tu_k-T\bar u\|_{W^{1,n}(\Omega_1)}+\|Tu_k-T\bar u\|_{C(\overline\Omega)}\to 0$ as $k\to\infty$.
\end{proof}

\section{ABP Maximum Principle}\label{sec:ABP}

In this section we assume that $f\in L^n(\Omega),\gamma\in L^\Psi_+(\Omega)$. We extend $f$ and $\gamma$ to ${\mathbb R}^n$ by setting $f(x)=\gamma(x)=0$ for $x\in {\mathbb R}^n\setminus \Omega$. Let for $\delta>0$, $\eta_\delta$ be the standard mollifier. Denote for $m=1,2,...$, $\bar\gamma_m:=\min(\gamma,m)$ and $
\bar\gamma_{m,\delta}=\eta_\delta*\bar\gamma_m$. Let $\varepsilon>0$ and let $\delta(m)$ be such that 
\begin{equation}\label{eq:gammatildemconv}
\int_{{\mathbb R}^n}|\bar\gamma_m-\bar\gamma_{m,\delta(m)}|^n\left(1+\left(\ln\left(1+|\bar\gamma_m-\bar\gamma_{m,\delta(m)}|\right)\right)^{n-1}\right)dx\leq \frac{1}{m}.
\end{equation}
We can assume that $\delta(m)\to 0$ as $m\to\infty$. We set $\gamma_m:=\bar\gamma_{m,\delta(m)}, \tilde\gamma_m:=\gamma-\gamma_m$. We begin with a technical lemma.

\begin{lemma}\label{lem:gammaconv}
$f\in L^n(\Omega),\gamma\in L^\Psi_+(\Omega)$. Then for every $\varepsilon>0$,
\begin{equation}\label{eq:gammamconv}
\int_{{\mathbb R}^n}|\tilde\gamma_m|^n\left(\ln\left(1+\frac{|\tilde\gamma_m|}{\varepsilon}\right)\right)^{n-1}dx\to 0\quad\mbox{as}\,\,m\to\infty.
\end{equation}
Moreover, if $U_R$ is a ball of radius $R>0$ and $U_{R+r}$ is the ball with the same center and radius $R+r$, $r>0$, then for sufficiently large $m$ we have
\begin{equation}\label{eq:gammamconv1}
\int_{U_R}|\gamma_m|^n\left(\ln\left(1+|\gamma_m|\right)\right)^{n-1}dx\leq \int_{U_{R+r}}|\gamma|^n\left(\ln\left(1+|\gamma|\right)\right)^{n-1}dx.
\end{equation}
Hence, we can consider the dependence of $C$ in \eqref{lem:est1} on $\gamma_m$ to be the same as the dependence of $C$ in \eqref{lem:est1} on $\gamma$.
\end{lemma}

\begin{proof}
Convergence \eqref{eq:gammamconv} is the consequence of the convexity of the function $s\to s^n(\ln(1+s))^{n-1}, s\geq 0$, dominated convergence theorem and \eqref{eq:gammatildemconv} since we have
\[
\begin{split}
&\int_{{\mathbb R}^n}|\tilde\gamma_m|^n\left(\ln\left(1+\frac{|\tilde\gamma_m|}{\varepsilon}\right)\right)^{n-1}dx
\leq 2^{n-1}\int_{{\mathbb R}^n}|\gamma-\bar\gamma_m|^n\left(\ln\left(1+\frac{2|\gamma-\bar\gamma_m|}{\varepsilon}\right)\right)^{n-1}dx
\\
&
\qquad\qquad+2^{n-1}\int_{{\mathbb R}^n}|\bar\gamma_m-\gamma_m|^n\left(\ln\left(1+\frac{2|\bar\gamma_m-\gamma_m|}{\varepsilon}\right)\right)^{n-1}dx
=:a_m+\frac{C(\varepsilon)}{m}\to 0.
\end{split}
\]

 Regarding \eqref{eq:gammamconv1}, using convexity of $s\to s^n(\ln(1+s))^{n-1}$ and Jensen's inequality, we have for $m$ such that $\delta(m)<r$,
\[
\begin{split}
\int_{U_R}|\gamma_m|^n\left(\ln\left(1+|\gamma_m|\right)\right)^{n-1}dx
&\leq \int_{U_R}\int_{B_{\delta(m)}}\eta_{\delta(m)}(y)|\gamma(x-y)|^n\left(\ln\left(1+|\gamma(x-y)|\right)\right)^{n-1}dx
\\
&
\leq \int_{U_{R+r}}|\gamma|^n\left(\ln\left(1+|\gamma|\right)\right)^{n-1}dx.
\end{split}
\]
\end{proof}

\begin{prop}\label{prop:strongsolexpterm}
Let $\Omega_2$ be a bounded domain which satisfies a uniform exterior cone condition. Let $f\in L^n(\Omega_2),\gamma\in L^\Psi_+(\Omega_2),{\rm supp}(\gamma)\Subset\Omega_2$. Let $\Omega_1$ be a domain with smooth boundary such that ${\rm supp}(\gamma)\Subset\Omega_1\Subset\Omega_2$. Then there exist $K=K(n,\lambda,\Lambda,\gamma,\|\gamma\|_{L^n(\Omega_2)},
{\rm dist}(\Omega_1,\partial\Omega_2), {\rm diam}(\Omega_2))>0$ and $\varepsilon=\varepsilon_0(n,\lambda,\Lambda,\gamma,\|\gamma\|_{L^n(\Omega_2)},
{\rm dist}(\Omega_1,\partial\Omega_2), {\rm diam}(\Omega_2))>0$ such that for every $0<\varepsilon\leq\varepsilon_0$ there exists a strong solution $u\in W^{2,n}_{\rm loc}(\Omega_2)\cap C(\overline \Omega_2)$ of
\begin{equation}\label{eq:Puccistrongexp}
\begin{cases}
{\mathcal P}^{\pm}(D^2u)\pm\gamma(x)|Du|=f(x)+\frac{K\varepsilon}{2}
\sum_{i=1}^n\exp\left(2^{\frac{n}{n-1}}\left(\frac{|u_{x_i}|}{K}\right)^{\frac{n}{n-1}}\right){\bf 1}_{\Omega_1}(x)\quad \mbox{in }\Omega_2,\\
u=0\quad\mbox{on }\partial \Omega_2
\end{cases}
\end{equation}
and such that 
\begin{equation}\label{eq:Pucciexpest1}
\|u\|_{C(\overline\Omega_2)}\leq C(\|f\|_{L^n(\Omega_2)}+\varepsilon),
\end{equation}
where $C=C(n,\lambda,\Lambda,\gamma,\|\gamma\|_{L^n(\Omega)},
{\rm dist}(\Omega_1,\partial\Omega_2), {\rm diam}(\Omega_2))$. The dependence of all the constants on $\gamma$ is the same as the dependence of $C$ on $\gamma$ in Lemma \ref {lem:strsolestimates}.
\end{prop}

\begin{proof} The proof is similar to the proof of Theorem \ref{th:strsol} and uses the Schauder Fixed Point Theorem.
We define
\[
B:=\{u\in W^{1,n}(\Omega_1)\cap C(\overline\Omega_2): \|u\|_{W^{2,n}(\Omega_1)}+\|u\|_{C(\overline\Omega_2)}\leq M, u=0\,\,\mbox{on}\,\,\partial \Omega_2\},
\]
where $M>0$ will be determined later. In the space $W^{1,n}(\Omega_1)\cap C(\overline\Omega_2)$ we consider the norm $\|u\|_{W^{1,n}(\Omega_1)}+\|u\|_{C(\overline\Omega_2)}$. 

Let $K= 2M (\frac{2n}{\alpha_n})^{\frac{n-1}{n}}$. Then, if $u\in B$, by \eqref{eq:MT1}
\begin{equation}\label{eq:L2n}
h_u:=\sum_{i=1}^n\exp\left(2^{\frac{n}{n-1}}\left(\frac{|u_{x_i}|}{K}\right)^{\frac{n}{n-1}}\right)\in L^{2n}(\Omega_1),\quad 
 \|h_u\|_{L^{2n}(\Omega_1)}\leq C(n,\Omega_1).
 \end{equation}
For $u\in B$ we consider the equation
\[
\begin{cases}
{\mathcal P}^-(D^2v)-\gamma(x)|Dv|=f(x)+K\varepsilon
\sum_{i=1}^n\exp\left(2^{\frac{n}{n-1}}\left(\frac{|u_{x_i}|}{K}\right)^{\frac{n}{n-1}}\right){\bf 1}_{\Omega_1}(x)\quad \mbox{in }\Omega_2,\\
v=0\quad\mbox{on }\partial \Omega_2.
\end{cases}
\]
Let $v\in W^{2,n}(\Omega_1)\cap C(\overline\Omega_2)$ be its unique strong solution given by Theorem \ref{th:strsol}. We denote it by $Tu$. 

{\bf Step 1.} $T:B\to B$ for an appropriate $M$. 

From Theorem \ref{prop:ABPnLp}, Lemma \ref {lem:strsolestimates} and \eqref{eq:L2n}, we have
\begin{equation}\label{eq:vw2nC}
 \begin{split}
 &
 \|v\|_{W^{2,n}(\Omega_1)}+\|v\|_{C(\overline\Omega_2)}
\\
&
\qquad
\le
C(n,\lambda,\Lambda,\gamma,\|\gamma\|_{L^n(\Omega_2)},
{\rm dist}(\Omega_1,\partial\Omega_2), {\rm diam}(\Omega_2))
\left(
\|f\|_{L^n(\Omega_2)}+K\varepsilon |\Omega_1|^{\frac{1}{2n}}C(n,\Omega_1)
\right).
\end{split}
\end{equation}
Thus, if we take $M=C(n,\lambda,\Lambda,\gamma,\|\gamma\|_{L^n(\Omega_2)},
{\rm dist}(\Omega_1,\partial\Omega_2), {\rm diam}(\Omega_2))
\left(
\|f\|_{L^n(\Omega)}+1
\right)$, and then $\epsilon_0$ such that $K\varepsilon_0 |\Omega_1|^{\frac{1}{2n}}C(n,\Omega_1)=1$, we have $Tu\in B$.

{\bf Step 2.} $T$ is continuous.

We begin exactly the same as in the proof of Step 2 of Theorem \ref{th:strsol}. Suppose $u,u_k\in B, k=1,2,...$ and $\|u-u_k\|_{W^{1,n}(\Omega_1)}+\|u-u_k\|_{C(\overline\Omega_2)}\to 0$ as $k\to\infty$. 
Since $Tu_k\in B$, we have that, up to a subsequence, $Tu_k\rightharpoonup w$ in $W^{2,n}(\Omega_1)$ for some $w$ such that $\|w\|_{W^{2,n}(\Omega_1)}\leq K$. Using compactness of the Sobolev embedding, this implies that $\|Tu_k-w\|_{W^{1,n}(\Omega_1)}\to 0$ as $k\to\infty$. 

It remains to show that $\|Tu_k-Tu\|_{C(\overline\Omega_2)}\to 0$. This will also give as that $Tu=w$ on $\Omega_1$. Up to a subsequence we have $Du_k\to Du$ a.e. in $\Omega_1$. Hence, since all functions $h_{u_k}$ satisfy \eqref{eq:L2n}, we have $\|h_u-h_{u_k}\|_{L^n(\Omega_1)}\to 0$ as $k\to\infty$.

Finally, since the above arguments apply for every subsequence of $u_k$, we conclude that $\|Tu_k-w\|_{W^{1,n}(\Omega_1)}+\|Tu_k-Tu\|_{C(\overline\Omega_2)}\to 0$ along the whole sequence.

{\bf Step 3.} $T(B)$ is compact.

We repeat the same arguments as in Step 3 of Theorem \ref{th:strsol}

Hence, by the Schauder Fixed Point Theorem, $T$ has a fixed point which is a solution of \eqref{eq:Puccistrongexp}. Estimate \eqref{eq:Pucciexpest1} follows from \eqref{eq:vw2nC}.
\end{proof}

\begin{theorem}\label{prop:ABPLnviscosity}
Let $\Omega$ be a bounded domain, $f\in L^n(\Omega),\gamma\in L^\Psi_+(\Omega)$. There exist $C=C(n,\lambda,\|\gamma\|_{L^n(\Omega)})$ such that if
$u\in C(\overline \Omega)$ is an $L^n$-viscosity subsolution (resp., supersolution) of 
\begin{equation}\label{eq:2n-Ln}
{\mathcal P}^-(D^2u)-\gamma(x)|Du|=f(x) \,\,\,\mbox{in}\,\,\Omega
\end{equation}
\[
\left(\mbox{resp.,}\,\,
{\mathcal P}^+(D^2u)+\gamma(x)|Du|=f(x) \,\,\,\mbox{in}\,\,\Omega
\right)
\]
then
\begin{equation}\label{eq:ABPnstatementLnvisc}
\max_{\overline \Omega} u\leq \max_{\partial\Omega} u+C{\rm diam}(\Omega)\|f_+\|_{L^n(\Gamma^+(u))}
\end{equation}
\[
\left(\mbox{resp.,}\,\,
\min_{\overline \Omega} u\geq \min_{\partial\Omega} u-C{\rm diam}(\Omega)\|f_-\|_{L^n(\Gamma^+(-u))}\right).
\]
\end{theorem}

\begin{proof}
We will only prove the statement for $L^n$-viscosity 
subsolutions. Without loss of generality we suppose that $0\in\Omega$. Let $\Omega_1, \Omega_2$ be respectively the open balls centered at $0$ with radii ${\rm diam}(\Omega)$ and $2{\rm diam}(\Omega)$. Let $\gamma_m, m\geq 1$ be from Lemma \ref{lem:gammaconv} and we assume that $m$ is sufficiently large so that ${\rm supp}(\gamma_m)\Subset\Omega_1$. Let $K,\varepsilon_0$ be from Proposition \ref{prop:strongsolexpterm} with these $\Omega_1, \Omega_2$ and $\gamma:=\gamma_m+c{\bf 1}_{\Omega_1}, 0<c<1$. We notice that it follows from Lemma \ref{lem:gammaconv} that the dependence of constants $K,\varepsilon_0$ on $\gamma_m$ is controlled by such dependence on $\gamma$ and is hence independent of $m$. Without loss of generality we can also assume that $f\geq 0$. Let $f_m\in C(\overline\Omega)$ be such that $f_m\geq 0$ in $L^n(\Omega)$, $f_m\to f$ in $L^n(\Omega)$ as $m\to\infty$.
Denote $\tilde f_m:=f-f_m$. We have for $0<\varepsilon<\varepsilon_0$,
\begin{align}
 \tilde\gamma_m |u_{x_i}|&=K\varepsilon\frac{\tilde\gamma_m}{\varepsilon}\frac{|u_{x_i}|}{K}\leq
K\varepsilon
\left(\exp\left(\left(\frac{|u_{x_i}|}{K}\right)^{\frac{n}{n-1}}\right)-1\right.
\nonumber\\
&+\left.C\left(\frac{|u_{x_i}|}{K}+\frac{\tilde\gamma_m}{\varepsilon}\left(\ln\left(1+\frac{\tilde\gamma_m}{\varepsilon}\right)\right)^{\frac{n-1}{n}}\right)\right)
{\bf 1}_{\Omega_1}.
\nonumber
\end{align}
Hence
\begin{align}
 \tilde\gamma_m|Du|&\leq
K\varepsilon
\sum_{i=1}^n\exp\left(\left(\frac{|u_{x_i}|}{K}\right)^{\frac{n}{n-1}}\right){\bf 1}_{\Omega_1}
\nonumber\\
&+C\left(\varepsilon |Du|{\bf 1}_{\Omega_1}+K\tilde\gamma_m\left(\ln\left(1+\frac{\tilde\gamma_m}{\varepsilon}\right)\right)^{\frac{n-1}{n}}\right).
\nonumber
\end{align}
Therefore, $u$ is an $L^n$-viscosity subsolution of
\begin{align}
{\mathcal P}^-(D^2u)-(\gamma_m(x)+C\varepsilon{\bf 1}_{\Omega_1}(x))|Du|&-K\varepsilon
\sum_{i=1}^n\exp\left(\left(\frac{|u_{x_i}|}{K}\right)^{\frac{n}{n-1}}\right){\bf 1}_{\Omega_1}(x)
\nonumber\\
&
=f(x)+CK\tilde\gamma_m(x)\left(\ln\left(1+\frac{\tilde\gamma_m(x)}{\varepsilon}\right)\right)^{\frac{n-1}{n}}.
\nonumber
\end{align}
We now require that $C\varepsilon <1$ (recall that $C$ above is the constant from \eqref{eq:dual}). We set $g_{\varepsilon,m}:=\tilde f_m(x)
+CK\tilde\gamma_m\left(\ln\left(1+\frac{\tilde\gamma_m}{\varepsilon}\right)\right)^{\frac{n-1}{n}}$. Let $w^m\in W^{2,n}_{\rm loc}(\Omega_2)\cap C(\overline \Omega_2)$ be the strong solution of
\begin{equation}\label{eq:Puccistrongexp1}
\begin{cases}
{\mathcal P}^{-}(D^2w^m)-(\gamma_m(x)+C\varepsilon{\bf 1}_{\Omega_1}(x))|Dw^m|&=g_{\varepsilon,m}(x)
\\
&+\frac{K\varepsilon}{2}
\sum_{i=1}^n\exp\left(2^{\frac{n}{n-1}}\left(\frac{|w^m_{x_i}|}{K}\right)^{\frac{n}{n-1}}\right){\bf 1}_{\Omega_1}(x)\quad \mbox{in }\Omega_2,\\
w^m=0\quad\mbox{on }\partial \Omega_2.
\end{cases}
\end{equation}
We have
\begin{equation}\label{eq:Pucciexpest11}
\|w^m\|_{C(\overline\Omega_2)}\leq C(\|g_{\varepsilon,m}\|_{L^n(\Omega_2)}+\varepsilon),
\end{equation}
where by Lemma \ref{lem:gammaconv}, $\|g_{\varepsilon,m}\|_{L^n(\Omega_2)}\to 0$ as $m\to\infty$.

By convexity, for every $a,b\in \mathbb R$
\[
\exp\left(\left(\frac{|a+b|}{K}\right)^{\frac{n}{n-1}}\right)\leq 
\frac{1}{2}\exp\left(2^{\frac{n}{n-1}}\left(\frac{|a|}{K}\right)^{\frac{n}{n-1}}\right)+\frac{1}{2}\exp\left(2^{\frac{n}{n-1}}\left(\frac{|b|}{K}\right)^{\frac{n}{n-1}}\right).
\]
It is then easy to see that $u^m:=u-w^m$ is an $L^n$-viscosity subsolution of
\[
{\mathcal P}^-(D^2u^m)-(\gamma_m(x)+C\varepsilon{\bf 1}_{\Omega_1}(x))|Du^m|-\frac{K\varepsilon}{2}
\sum_{i=1}^n\exp\left(2^{\frac{n}{n-1}}\left(\frac{|u_{x_i}^m|}{K}\right)^{\frac{n}{n-1}}\right){\bf 1}_{\Omega_1}(x)
=f_m(x)
\]
and hence an $L^n$-viscosity subsolution of
\[
{\mathcal P}^-(D^2u^m)-(\gamma_m(x)+C\varepsilon)|Du^m|-\frac{K\varepsilon}{2}
\sum_{i=1}^n\exp\left(2^{\frac{n}{n-1}}\left(\frac{|u_{x_i}^m|}{K}\right)^{\frac{n}{n-1}}\right)
=f_m(x).
\]
Since $\gamma_m,f_m$ are continuous, $u^m$ is also a $C$-viscosity solution of the above equation.
We can now basically repeat the arguments from \cite{KS-JMPA}. We extend $u^m$ to $\mathbb R^n$ so that $u^m\in BUC(\mathbb R^n)$.
For $\delta>0$ we define the sup-convolution of $u^m$ by
\[
u^m_\delta(x):=\sup\left\{ u^m(y)-\frac {|x-y|^2}{2\delta} \ : \ y\in\mathbb R^n\right\}.
\]
It is well known that $u^m_\delta$ satisfies
\begin{equation}\label{eq:lipcontumdelta}
|u^m_\delta(x)-u^m_\delta(y)|\leq \frac{2\sqrt{\|u^m\|_{\infty}}}
{\sqrt \delta} |x-y|
\leq \frac{2\sqrt{\|u\|_{\infty}+1}}
{\sqrt \delta} |x-y|\quad \mbox{for }x,y\in {\mathbb R}^n.
\end{equation}
Moreover, $u^m_\delta$ is a $C$-viscosity subsolution of
\[
{\mathcal P}^-(D^2u^m_\delta)-(\gamma_m(x)+\rho_1(\delta)+C\varepsilon)|Du^m_\delta|-\frac{K\varepsilon}{2}
\sum_{i=1}^n\exp\left(2^{\frac{n}{n-1}}\left(\frac{|(u_\delta^m)_{x_i}|}{K}\right)^{\frac{n}{n-1}}\right)
=f_m(x)+\rho_2(\delta)
\]
in $\Omega_\delta=\{x\in\Omega:{\rm dist}(x,\partial\Omega)>2\sqrt{\delta(\|u\|_{\infty}+1)}\}$ for small $\delta$, where $\rho_1(\delta),\rho_2(\delta)\geq 0$ are such that $\rho_1(\delta)+\rho_2(\delta)\to 0$ as $\delta\to 0$. Since the functions $u^m_\delta$ are semiconvex and hence pointwise twice differentiable a.e. and for a.e. $x\in\Omega_\delta$,
\begin{equation}\label{eq:umdelta}
{\mathcal P}^-(D^2u^m_\delta(x))-(\gamma_m(x)+\rho_1(\delta)+C\varepsilon)|Du^m_\delta(x)|-\frac{K\varepsilon}{2}
\sum_{i=1}^n\exp\left(2^{\frac{n}{n-1}}\left(\frac{|(u_\delta^m)_{x_i}(x)|}{K}\right)^{\frac{n}{n-1}}\right)
\leq f_m(x)+\rho_2(\delta).
\end{equation}
Without loss of generality we can assume that the sets $\Omega_\delta$ are domains. We could now proceed as in \cite{KS-JMPA} by taking the inf-convolutions, however we want to show an alternative way by mollification. 

Let $\eta_\tau$ be standard mollifiers in $\mathbb R^n$ for $\tau>0$. Denote $u^m_{\delta,\tau}:=\eta_\tau*u^m_\delta$. Multiplying \eqref{eq:umdelta} by $\eta_\tau(y-x)$, integrating both sides, using Jensen's inequality and that the operator in the left hand side of \eqref{eq:umdelta} is concave, using \eqref{eq:lipcontumdelta}, continuity of $\gamma_m,f_m$, and the fact that if $D^2u^m_\delta$ denotes the absolutely continuous part of the second order distributional derivative of $u^m_\delta$, then for every $y$ we have the matrix inequality, $D^2u^m_{\delta,\tau}(y)\geq \eta_\tau*D^2u^m_\delta(y)$, we obtain that for every sufficiently small $\tau>0$ and $y\in \Omega_{2\delta}$
\[
\begin{split}
{\mathcal P}^-(D^2u^m_{\delta,\tau}(y))-(\gamma_m(y)+2\rho_1(\delta)+C\varepsilon)|Du^m_{\delta,\tau}(y)|&\leq \frac{K\varepsilon}{2}
\sum_{i=1}^n\exp\left(2^{\frac{n}{n-1}}\left(\frac{|(u_{\delta,\tau}^m)_{x_i}(y)|}{K}\right)^{\frac{n}{n-1}}\right)
\\
&
+ f_m(y)+2\rho_2(\delta).
\end{split}
\]
Hence, by choosing $\delta=1/k, k\geq 1$ and sufficiently small $\tau=\tau(1/k)$ we have that the sequence $v^m_k:=u^m_{1/k,\tau(1/k)}\in C^2(\Omega)$ converges uniformly to $u^m$ on $\Omega$ and $v^m_k$ satisfies in $\Omega_{2/k}$
\begin{equation}\label{eq:umdelta1}
\begin{split}
{\mathcal P}^-(D^2v^m_k(y))-(\gamma_m(y)+a_k+C\varepsilon)|Dv^m_k(y)|&\leq\frac{K\varepsilon}{2}
\sum_{i=1}^n\exp\left(2^{\frac{n}{n-1}}\left(\frac{|(v_k^m)_{x_i}(y)|}{K}\right)^{\frac{n}{n-1}}\right)
\\
&
+f_m(y)+a_k,
\end{split}
\end{equation}
where $0\leq a_k\to0$ as $k\to\infty$.
We set
\[
d_{m,k}=\frac{\max_{\overline \Omega_{2/k}} v^m_\delta-\max_{\partial\Omega_{2/k}}v^m_k}{{\rm diam}(\Omega_{2/k})}
\]
and notice that, since $v^m_k$ converge uniformly to $u^m$ as $k\to \infty$ and $u^m$ converge uniformly to $u$ as $m\to\infty$ on $\Omega$, there is a constant $M>0$ such that $d_{m,k}\leq M$ for all $k,m$.
For a domain $D\subset \mathbb R^n$, 
a function $v:D\to \mathbb R$, and 
$r>0$, we denote
$$
\Gamma^+_r(v,D):=\{ x\in D \ | \ \exists p\in B_r(0)\mbox{ such that }v(y)\leq v(x)+\langle p,y-x\rangle \mbox{ for }y\in D\}.
$$
We have
\[
B_{d_{m,k}}(0)=Dv^m_\delta(\Gamma^+_{d_{m,k}}(v^m_k,\Omega_{2/k})).
\]
Moreover we notice that if 
$y\in \Gamma^+_{d_{m,k}}(v^m_k,\Omega_{2/k})$ then $|Dv^m_k(y)|<M$ and thus 
\[
\frac{K\varepsilon}{2}\sum_{i=1}^n\exp\left(2^{\frac{n}{n-1}}\left(\frac{|(v_k^m)_{x_i}(y)|}{K}\right)^{\frac{n}{n-1}}\right)\leq L\varepsilon
\]
for some absolute constant $L\geq 0$.
It now follows from
the proof of ABP maximum principle (see \cite{GT}, Section 9.1 or \cite{CCKS}, Appendix A) that 
\[
\max_{\overline \Omega_{2/k}} v^m_k\leq \max_{\partial\Omega_{2/k}} v^m_k+C{\rm diam}(\Omega)(\|f_m\|_{L^n(\Gamma^+_{d_{m,k}}(v^m_k,\Omega_{2/k}))}+(L\varepsilon+a_k)
|\Omega|^{\frac{1}{n}}),
\]
where $C=C(n,\lambda,\|\gamma_m+a_k+C\varepsilon\|_{L^n(\Omega_{2/k})})$. 

We now apply Lemma A.1(i) of \cite{CCKS} which guarantees that
$\limsup_{k\to\infty} \Gamma^+(v^m_k,\Omega_{2/k})\subset \Gamma^+(v^m,\Omega)$ to obtain
\[
\max_{\overline \Omega} v^m\leq \max_{\partial\Omega} v^m+C{\rm diam}(\Omega)(\|f_m\|_{L^n(\Gamma^+(v^m,\Omega))}+L\varepsilon|\Omega|^{\frac{1}{n}}),
\]
where $C=C(n,\lambda,\|\gamma_m+C\varepsilon\|_{L^n(\Omega)})$
We then send $m\to\infty$ and use the uniform convergence of $v_m$ to $u$ on $\Omega$, the fact that $\|\gamma_m-\gamma\|_{L^n(\Omega)}+ 
\|f_m-f\|_{L^n(\Omega)}\to 0$, and again Lemma A.1(i) of \cite{CCKS} to obtain
\[
\max_{\overline \Omega} u\leq \max_{\partial\Omega} u+C{\rm diam}(\Omega)(\|f\|_{L^n(\Gamma^+(u,\Omega))}+L\varepsilon|\Omega|^{\frac{1}{n}}),
\]
where $C=C(n,\lambda,\|\gamma\|_{L^n(\Omega)})$. It now remains to send $\varepsilon\to 0$ to get
\eqref{eq:ABPnstatementLnvisc}.

\end{proof}

\end{document}